\documentclass[11pt]{article}

\usepackage[margin=1.3in]{geometry}
\usepackage{amsmath, amsfonts, amssymb, amsthm}
\usepackage{graphicx}
\usepackage{hyperref}
\hypersetup{
    colorlinks,%
    citecolor=black,%
    filecolor=black,%
    linkcolor=black,%
    urlcolor=black
}
\usepackage{color}
\usepackage{mathrsfs}
\usepackage{textcomp}
\usepackage{color}
\usepackage{marginnote}

\usepackage{verbatim}

\usepackage{sectsty}

\makeatletter

\newdimen\bibspace
\renewenvironment{thebibliography}[1]{%
 \section*{\refname 
       \@mkboth{\MakeUppercase\refname}{\MakeUppercase\refname}}%
     \list{\@biblabel{\@arabic\c@enumiv}}%
          {\settowidth\labelwidth{\@biblabel{#1}}%
           \leftmargin\labelwidth
           \advance\leftmargin\labelsep
           \itemsep\bibspace
           \parsep\z@skip     %
           \@openbib@code
           \usecounter{enumiv}%
           \let\p@enumiv\@empty
           \renewcommand\theenumiv{\@arabic\c@enumiv}}%
     \sloppy\clubpenalty4000\widowpenalty4000%
     \sfcode`\.\@m}
    {\def\@noitemerr
      {\@latex@warning{Empty `thebibliography' environment}}%
     \endlist}

\makeatother

\makeatletter

\newtheorem{thm}{Theorem}[section]
\newtheorem{lem}[thm]{Lemma}

\newtheorem*{rem*}{Remark}

\def\XXint#1#2#3{{\setbox0=\hbox{$#1{#2#3}{\int}$}
  \vcenter{\hbox{$#2#3$}}\kern-.5\wd0}}

\newcommand{\om}{\Omega}                
\newcommand{\va}{\varepsilon}           \newcommand{\ud}{\mathrm{d}}
\newcommand{\be}{\begin{equation}}      \newcommand{\ee}{\end{equation}}

\newcommand{\R}{\mathbb{R}}

\begin{document}

\title{\textbf{Symmetric radial decreasing rearrangement can increase the fractional Gagliardo norm in domains}
\bigskip}

\author{Dong Li \footnote{DL is partially supported by Hong Kong RGC grant GRF 16307317.}, \quad  Ke Wang\footnote{KW is partially supported by HKUST Initiation Grant IGN16SC05.}}

\date{\today}

\maketitle

\begin{abstract}
We show that the symmetric radial decreasing rearrangement can increase the fractional Gagliardo semi-norm in domains.
\end{abstract}


\section{Introduction}
For any Borel set $A$ in $\mathbb R^n$ with $|A|<\infty$ ($|A|$ denotes the Lebesgue
measure of $A$), define $A^{\ast}$, the
symmetric rearrangement of $A$ as the open ball
\begin{align*}
A^{\ast} = \{x:\, |x| < ( |A|/ \alpha_n)^{\frac 1n} \},
\end{align*}
where $\alpha_n = \pi^{\frac n2}/\Gamma(\frac n2+1)$ is the volume of the unit ball. 
If $|A|=0$, then $A^{\ast} = \varnothing$ and for later purposes we 
conveniently define $\chi_{\varnothing}\equiv 0$.
Denote by $\mathscr U_0$ the space of Borel measurable functions
$u:\, \mathbb R^n \to \mathbb R$ such that
\begin{align} \notag 
\mu_u(t)= |\{ x: \; |u(x)| >t \}| \ \ \mbox{is finite for all }t>0.
\end{align}
Observe that $\mu_u(\cdot)$ is right-continuous, non-increasing and (by the Lebesgue
dominated convergence theorem)
$\lim_{t\to \infty} \mu_u(t)=0$. 
For any $u \in \mathscr U_0$, define the symmetric decreasing rearrangement $u^{\ast}$
as 
\begin{align*}
u^{\ast}(x) =\int_0^{\infty} \chi_{ \{|u|>t \}^{\ast} }(x) dt
=\sup\{t: |\{ |u|>t\}|>\alpha_n|x|^n\}.
\end{align*}
Since $\mu_u$ decays to zero as $t\to \infty$, 
we have $0\le u^{\ast}(x) <\infty$ for any $x \ne 0$, 
whereas $u^{\ast}(0)$ may be $\infty$.
Evidently, the function $u^{\ast}$ is radial, non-increasing in $|x|$, and satisfy
\begin{align*}
\{ |u|>t  \}^{\ast} = \{ u^{\ast}> t \}, \quad \forall\, t>0.
\end{align*}
From this one can deduce $|\{|u|>t\}|= |\{ u^{\ast} >t\}|$, $\forall\, t>0$ and $\| u^{\ast} \|_p
= \| u \|_p$ for all $1\le p\le \infty$. Note that it follows from the level set characterisation that any
uniform translation of $u$ does not change $u^{\ast}$,
namely if for any $x_0\in \mathbb R^n$, we define $u_{x_0}(x)= u(x-x_0)$, then
\begin{align}\label{eq:syminvariant}
(u_{x_0})^{\ast} = u^{\ast}.
\end{align}
This simple property will be used without explicit mentioning later.
 On the other hand, the effect of rearrangement on the
gradient of the function is more complex and interesting.
Let $u$ be a nonnegative smooth function that vanishes at infinity. The P\'olya-Szeg\"o \cite{PS} inequality states that for $1\le p<\infty$, 
\[
\int_{\R^n}|\nabla u|^p\ge \int_{\R^n}|\nabla u^*|^p.
\]
Brothers-Ziemer \cite{BZ} gave a characterization of the equality case under the assumption that the distribution function of $u$ is absolutely continuous. 
This P\'olya-Szeg\"o inequality also holds for every bounded open set $\Omega\subset \R^n$. That is, for every nonnegative $u\in C^\infty_c(\Omega)$, we also have
\[
\int_{\Omega}|\nabla u|^p\ge \int_{\Omega^*}|\nabla u^*|^p.
\]
As a matter of fact, one can show that for every
$u\in W_0^{1,p}(\Omega)$, one has $u^{\ast} \in W_0^{1,p}(\Omega^{\ast})$ and the above
inequality holds.

We are interested in the effect of symmetric decreasing rearrangement for fractional Sobolev inequalities. For $0<\sigma<1$ and $1\le p<\infty$, we define the space $\mathring W^{\sigma,p}(\Omega)$ as the completion of $C^\infty_c(\Omega)$ under the norm
\[
\|u\|_{\mathring W^{\sigma,p}(\om)}=\biggl({\iint_{\om\times \om} \frac{|u(x)-u(y)|^p}{|x-y|^{n+\sigma p}}\,\ud x\ud y} \biggr)^{\frac 1p}.
\]
It was shown in Theorem 9.2 in Almgren-Lieb \cite{AL} that 
\[
\|u\|_{\mathring W^{\sigma,p}(\R^n)}\ge \|u^*\|_{\mathring W^{\sigma,p}(\R^n)}.
\]
Characterizations of the equality case have been given in Burchard-Hajaiej \cite{BH} and Frank-Seiringer \cite{FS}. Motivated by the P\'olya-Szeg\"o inequality in domains, we would like to investigate whether the above inequality holds for bounded open sets $\Omega$. That is, do we have
\begin{equation}\label{eq:rearrangement1}
\|u\|_{\mathring W^{\sigma,p}(\Omega)}\ge \|u^*\|_{\mathring W^{\sigma,p}(\Omega^*)}?
\end{equation}

Another motivation of the above question  comes from Frank-Jin-Xiong \cite{FJX}, where the authors study the best constants of fractional Sobolev inequalities on domains.  A classical result of Lieb \cite{Lieb} implies that 
\be \label{eq:FSI-2}
S(n,\sigma, \R^n)\left(\int_{\R^n} |u|^{\frac{2n}{n-2\sigma}}\,\ud x\right)^{\frac{n-2\sigma}{n}} \le\|u\|^2_{\mathring W^{\sigma,2}(\R^n)}  \quad \mbox{for all }u\in \mathring W^{\sigma,2}(\R^n),
\ee
where $S(n,\sigma,\R^n)=\frac{2^{1-2\sigma}\omega_n^{\frac{2\sigma}{n}}\pi^{\frac{n}{2}}\Gamma(2-\sigma)}{\sigma(1-\sigma) \Gamma(\frac{n-2\sigma}{2})}$ and $\omega_n$ is the volume of the unit $n-$dimensional sphere.  Moreover, the equality in \eqref{eq:FSI-2} holds if and only if
$
u(x)= (1+|x|^2)^{-\frac{n-2\sigma}{2}}
$
up to translating and scaling. These follow from the fact that the sharp fractional Sobolev inequality is a dual inequality of the sharp Hardy-Littlewood-Sobolev inequality. For an open set $\Omega\neq\R^n$, if $\sigma\in (1/2, 1)$ and $n\ge 2$,
then there exists a positive constant  $\underline{S}(n,\sigma)$ depending only on $n,\sigma$ but \emph{not} on $\Omega$ such that
\be \label{eq:FSI}
\underline{S}(n,\sigma)\left(\int_{\om} |u|^{\frac{2n}{n-2\sigma}}\,\ud x\right)^{\frac{n-2\sigma}{n}} \le \iint_{\om\times\Omega} \frac{(u(x)-u(y))^2}{|x-y|^{n+2\sigma}}\,\ud x\ud y \quad \mbox{for all }u\in \mathring W^{\sigma,2}(\om).
\ee
This inequality is called the fractional Sobolev inequality in domain $\Omega$. It is included in Theorem 1.1 in Dyda-Frank \cite{DF}. It actually follows from \eqref{eq:FSI-2} and a fractional Hardy inequality of Dyda \cite{Dyda}, Loss-Sloane \cite{LS} and Dyda-Frank \cite{DF} (by using similar arguments to the proof of Theorem \ref{thm:rearrangementestimate} here; see the remark in the end of this paper). 

In Frank-Jin-Xiong \cite{FJX}, they studied the best constant in \eqref{eq:FSI}:  
\[
S(n,\sigma,\Omega):=\inf\left\{\iint_{\Omega\times\Omega} \frac{(u(x)-u(y))^2}{|x-y|^{n+2\sigma}}\,\ud x\ud y\ |\  u\in C_c^\infty(\om), \int_{\om} |u|^{\frac{2n}{n-2\sigma}}\,\ud x=1\right\}.
\] 
It was proved in \cite{FJX} that this best constant $S(n,\sigma,\om)$ actually depends on the domain $\Omega$, and can be achieved in many cases such as in the half spaces $\R^n_+=\{x=(x',x_n)\in \R^n, x_n>0\}$ or some smooth bounded domains, which is in contrast to the classical Sobolev inequalities in domains. Let $B_r$ be the  ball of radius $r$ centered at the origin, and $B^+_1=B_1\cap\R^n_+$. Suppose $\sigma\in (1/2,1)$, and $\Omega$ is a $C^2$ bounded open set such that $B_1^+\subset\Omega\subset \R^n_+$, then it was proved in \cite{FJX} that both $S(n,\sigma,\R^n_+)$ and $S(n,\sigma,\om)$ are achieved, and there holds the inequality
\[
S(n,\sigma,\om)<S(n,\sigma,\R^n_+)<S(n,\sigma,\R^n).
\]
On the other hand, from \eqref{eq:FSI}, we have that for $\sigma\in (1/2, 1)$,
$S(n,\sigma,\Omega)\ge \underline{S}(n,\sigma)>0$ for every open set $\Omega$. An interesting question left open is to find the value of $\inf_{\Omega} S(n,\sigma,\Omega)$ for $\sigma\in (1/2,1)$, where the infimum is taken over all bounded open sets $\Omega$. A conjecture is that $\inf_{\Omega} S(n,\sigma,\Omega)$ is achieved by a ball, which could follow from \eqref{eq:rearrangement1} . However, we show in this paper that $\eqref{eq:rearrangement1}$ is false.


\begin{thm}\label{thm:rearrangement2}
Let $n\ge 1$ and $\Omega $ be any nonempty open set in $\mathbb R^n$ with $|\Omega| < \infty$. 
Let $\sigma\in (0,1)$ and $p\in(0,\infty)$. There exists a nonnegative $u\in C^\infty_c(\Omega)$ such that
\[
\iint_{\Omega \times \Omega} \frac{|u(x)-u(y)|^p}{|x-y|^{n+\sigma p}}\,\ud x\ud y
< \iint_{\Omega^*\times \Omega^*} \frac{|u^*(x)-u^*(y)|^p}{|x-y|^{n+\sigma p}}\,\ud x\ud y.
\]
\end{thm}
We will prove this theorem in the next section by using an explicit computation.

\begin{rem*}\label{rem:unitballfirst}
Theorem \ref{thm:rearrangement2} holds in particular when $\Omega$ is an open ball centered at the origin  (so that $\Omega^*=\Omega$). 
\end{rem*}

\begin{rem*}
In \cite{AL} (see Corollary 2.3 therein), a general rearrangement inequality is shown to hold for
convex integrands. Namely, if $\Psi:\mathbb R^+ \to \mathbb R^+$ is convex with
$\Psi(0)=0$, then for every nonnegative $L^1(\mathbb R^n)$ function $W$, every 
nonnegative $f$, $g \in \mathscr U_0$ 
with $\Psi \circ f $, $\Psi \circ g \in L^1(\mathbb R^n)$, one has
\begin{align*}
\int_{\mathbb R^n}
\int_{\mathbb R^n} \Psi(|f(x) -g (y) |) W(x-y) dx dy \ge \int_{\mathbb R^n}
\int_{\mathbb R^n} \Psi(|f^*(x)-g^*(y)|) W^*(x-y)dx dy.
\end{align*}
Our Theorem \ref{thm:rearrangement2} shows that such a general result cannot hold
if $\mathbb R^n$ is replaced by a domain $\Omega$ of finite measure on the left-hand side (and correspondingly by
$\Omega^*$ on the right-hand side).
\end{rem*}




On the other hand, we have the following estimate.

\begin{thm}\label{thm:rearrangementestimate}
Let $n\ge 1$, $\sigma\in (0,1)$ and $p\in(0,\infty)$ be such that $\sigma p>1$. Then there exists a positive constant $C$ depending only on $n,\sigma$ and $p$ such that
\[
\iint_{\R^n\times \R^n} \frac{|u^*(x)-u^*(y)|^p}{|x-y|^{n+\sigma p}}\,\ud x\ud y\le C\iint_{\Omega \times \Omega} \frac{|u(x)-u(y)|^p}{|x-y|^{n+\sigma p}}\,\ud x\ud y
\]
for all open sets $\Omega \subset \mathbb R^n$ and all nonnegative $u\in C^\infty_c(\Omega)$. In particular,
\[
\iint_{\Omega^*\times \Omega^*} \frac{|u^*(x)-u^*(y)|^p}{|x-y|^{n+\sigma p}}\,\ud x\ud y\le C\iint_{\Omega \times \Omega} \frac{|u(x)-u(y)|^p}{|x-y|^{n+\sigma p}}\,\ud x\ud y.
\]

\end{thm}

\section{Proofs}\label{sec:proof}

We begin with the following simple lemma. Recall that for any two sets $A$ and $B$, their symmetric difference $A\triangle B=(A\setminus B)\cup (B\setminus A)$.
\begin{rem*}
For an open set $\Omega\subset\R^n$, $|\Omega^*\triangle\Omega|=0$ if and only if $\Omega=\Omega^*$. 
\end{rem*}

\begin{lem}\label{lem:decrease}
Let $\Omega$ be an open and bounded set in $\R^n$.
\begin{itemize}
\item[(i).] Suppose $f\in L^1_{\operatorname{loc}}(\R^n)$ is radial and strictly decreasing, i.e. $f(x)>f(y)$ if $|x|<|y|$. Then
\begin{align}\label{eq:lemma1}
\int_{\Omega^*}f(x)\,\ud x>\int_{\Omega}f(x)\,\ud x\quad\mbox{if}\ \ |\Omega^*\triangle\Omega|>0.
\end{align}

\item[(ii).] Suppose $\overline B_\delta\subset\Omega$ for some $\delta>0$, and let $f\in L^1(\R^n\setminus\overline B_\delta)$ be radial and strictly decreasing. Then
\begin{align}\label{eq:lemma2}
\int_{\R^n\setminus\Omega}f(x)\,\ud x>\int_{\R^n\setminus\Omega^*}f(x)\,\ud x\quad\mbox{if}\ \ |\Omega^*\triangle\Omega|>0.
\end{align}
\end{itemize}
\end{lem}

\begin{rem*}
The main example is $f(x)=|x|^{-\alpha}$ for some $\alpha>0$.  Similar proof as below
can show the well-known inequality that for any Borel measure $A\subset \mathbb R^n$ with
$|A|<\infty$, $x_0\in \mathbb R^n$, and $\delta>0$,  one has
\begin{align*}
\int_{\mathbb R^n\setminus A} |x-x_0|^{-n-\delta} dx \ge 
\int_{\mathbb R^n \setminus A^{*} } |x|^{-n-\delta} dx = \operatorname{const} \cdot |A|^{-\frac {\delta}n}.
\end{align*}
This inequality can be used to establish fractional Sobolev embedding. We should stress that in our
case one needs strict inequality and for this reason we impose strict monotonicity on $f$.
\end{rem*}

\begin{proof}
Let $r$ be the radius of $\Omega^*$. 

We prove $(i)$ first. Notice that
\[
\int_{\Omega^*}f(x)\,\ud x-\int_{\Omega}f(x)\,\ud x=\int_{\Omega^*\setminus\Omega}f(x)\,\ud x-\int_{\Omega\setminus\Omega^*}f(x)\,\ud x.
\]
Since $|\Omega^*|=|\Omega|$, we have $|\Omega\setminus\Omega^*|=|\Omega^*\setminus\Omega|=\frac 12 |\Omega^*\triangle\Omega|>0$. Since $f$ is radial and strictly decreasing, we have
\[
\begin{split}
\int_{\Omega^*\setminus\Omega}f(x)\,\ud x&>f(r)|\Omega^*\triangle\Omega|,\\
\int_{\Omega\setminus\Omega^*}f(x)\,\ud x&<f(r)|\Omega^*\triangle\Omega|.
\end{split}
\]
Hence, the inequality \eqref{eq:lemma1} follows.

To prove $(ii)$, we notice $\overline B_\delta\subset\Omega^*$ by the assumption, and
\[
\int_{\R^n\setminus\Omega}f(x)\,\ud x-\int_{\R^n\setminus\Omega^*}f(x)\,\ud x=\int_{\Omega^*\setminus\Omega}f(x)\,\ud x-\int_{\Omega\setminus\Omega^*}f(x)\,\ud x.
\]
Hence, the inequality \eqref{eq:lemma2} follows the same as above.
\end{proof}

\begin{proof}[Proof of Theorem \ref{thm:rearrangement2}] 
Our proof of the general case in Theorem \ref{thm:rearrangement2} is inspired by that of the special case $\Omega$ being a ball. So we will provide the proof of Theorem \ref{thm:rearrangement2} for $\Omega=B_1$ first. 

Let $\eta \in C_c^{\infty}(B_1)$ be a radially decreasing function such that $\eta(x)=1$ for $|x| \le 1/2$. Let $\va\in (0,1/2)$ which will be chosen very small, $$x_\va=(1-\va, 0,\cdots,0),$$ and
\[
u_\va=\eta\left(\frac{x-x_\va}{\va}\right).
\]
Since we assumed that $\eta$ is smooth, nonnegative, and radially decreasing, it is clear that
\[
u^*_\va=\eta\left(\frac{x}{\va}\right).
\]
Therefore,  
\begin{equation}\label{eq:equal}
\iint_{\R^n\times \R^n} \frac{|u_\va(x)-u_\va(y)|^p}{|x-y|^{n+ \sigma p}}\,\ud x\ud y= \iint_{\R^n\times \R^n} \frac{|u_\va^*(x)-u_\va^*(y)|^p}{|x-y|^{n+\sigma p}}\,\ud x\ud y.
\end{equation}
Since $B_1^*=B_1$ and
\begin{align*}
&\iint_{B_1\times B_1} \frac{|u_\va(x)-u_\va(y)|^p}{|x-y|^{n+ \sigma p}}\,\ud x\ud y\\
&=\iint_{\R^n\times \R^n} \frac{|u_\va(x)-u_\va(y)|^p}{|x-y|^{n+\sigma p}}\,\ud x\ud y-2\int_{B_1}u_\va^p(x)\left(\int_{\R^n\setminus B_1}\frac{1}{|x-y|^{n+\sigma p}}\,\ud y\right)\ud x,
\end{align*}
we only need to show that
\begin{equation}\label{eq:aux0}
\int_{B_1}u_\va^p(x)\left(\int_{\R^n\setminus B_1}\frac{1}{|x-y|^{n+\sigma p}}\,\ud y\right)\ud x>\int_{B_1}(u_\va^*(x))^p\left(\int_{\R^n\setminus B_1}\frac{1}{|x-y|^{n+\sigma p}}\,\ud y\right)\ud x.
\end{equation}

First, since $u^*$ is supported in $B_\va$, we have
\begin{align*}
\int_{B_1}(u_\va^*(x))^p\left(\int_{\R^n\setminus B_1}\frac{1}{|x-y|^{n+\sigma p}}\,\ud y\right)\ud x
&=\int_{B_\va}(u_\va^*(x))^p\left(\int_{\R^n\setminus B_1}\frac{1}{|x-y|^{n+\sigma p}}\,\ud y\right)\ud x\\
&\le C \int_{B_\va}(u_\va^*(x))^p\,\ud x= C\va^n\int_{B_1}\eta^p(x)\,\ud x,
\end{align*}
where (as well as in the below) $C$ is a positive constant independent of $\va$.

Secondly, since $u$ is supported in $B_\va(x_\va)$, we have
\begin{align*}
\int_{B_1}u_\va^p(x)\left(\int_{\R^n\setminus B_1}\frac{1}{|x-y|^{n+\sigma p}}\,\ud y\right)\ud x
&\ge \int_{B_\va(x_\va)}u_\va^p(x)\left(\int_{\{y:\ y_1\ge 1\}}\frac{1}{|x-y|^{n+\sigma p}}\,\ud y\right)\ud x\\
&\ge C \int_{B_\va(x_\va)}\frac{u_\va^p(x)}{(1-x_1)^{\sigma p}}\,\ud x\\
&\ge C\va^{-\sigma p}\int_{B_\va(x_\va)}u_\va^p(x)\,\ud x=C\va^{n-\sigma p}\int_{B_1}\eta^p(x)\,\ud x,
\end{align*}
where in the second inequality, we used that for $x=(x_1,\cdots,x_n)$ with $x_1<1$, 
\[
\int_{\{y: \ y_1\ge 1\}}\frac{1}{|x-y|^{n+\sigma p}}\,\ud y= C(1-x_1)^{-\sigma p}.
\]
This proves \eqref{eq:aux0}, and thus Theorem \ref{thm:rearrangement2} for $\Omega=B_1$, if we choose $\va$ sufficiently small. 

Now let us consider the general case where $\Omega$ is not a ball.  Since $\Omega$ is an open set and the Gagliardo semi-norm is translation invariant and dilation invariant (and also by \eqref{eq:syminvariant}), without loss of generality, we may assume $\Omega$ contains $B_1$. Again,  let $\eta \in C_c^{\infty}(B_1)$ be a radially decreasing function such that $\eta(x)=1$ for $|x| \le 1/2$. Define for $\va\in (0,1)$,
\begin{align*}
u_{\va} (x)=\eta\left(\frac {x} {\va}\right).
\end{align*}
Hence,
\begin{align*}
u_{\va}^{\ast}(x) =\eta\left(\frac x {\va}\right),
\end{align*}
and thus, \eqref{eq:equal} also holds.
Since 
\begin{align*}
&\iint_{\Omega\times \Omega} \frac{|u_{\va}(x)-u_{\va}(y)|^p}{|x-y|^{n+ \sigma p}}\,\ud x\ud y\\
&=\iint_{\R^n\times \R^n} \frac{|u_{\va}(x)-u_{\va}(y)|^p}{|x-y|^{n+\sigma p}}\,\ud x\ud y-2\int_{\Omega}u_{\va}^p(x)\left(\int_{\R^n\setminus \Omega}\frac{1}{|x-y|^{n+\sigma p}}\,\ud y\right)\ud x,
\end{align*}
we only need to check the inequality
\begin{equation}\label{eq:aux}
\int_{\Omega} u^p_{\va}(x) F(x) dx 
> \int_{\Omega^*} (u^{\ast}_{\va}(x))^p \tilde F(x) dx,
\end{equation}
where 
\begin{equation}\label{eq:F}
F(x) = \int_{\R^n\setminus\Omega} \frac 1 { |x-y|^{n+p \sigma}} dy\quad\mbox{and}\quad \tilde F(x) = \int_{\R^n\setminus\Omega^*} \frac 1 { |x-y|^{n+p \sigma}} dy.
\end{equation}
Noticing the support of $\eta$, this reduces to checking the inequality
\begin{align}\label{eq:Finequality}
\int_{B_1} \eta^p(x) F(\va x) dx > \int_{B_1} \eta^p(x) \tilde F(\va x) dx.
\end{align}
Since $\Omega$ is an open and is not a ball, we have $|\Omega^*\setminus\Omega|=|\Omega\setminus\Omega^*|>0$. Then it follows from \eqref{eq:lemma2} in Lemma \ref{lem:decrease} that  $F(0)>\tilde F(0)$. Hence, the inequality \eqref{eq:Finequality} holds for all $\va$ sufficiently small by using the Lebesgue dominated convergence theorem. Theorem \ref{thm:rearrangement2} is proved.

We remark that the above proof for the general case where $\Omega$ is not a ball can also be used to prove the case when $\Omega=B_1$, which is as follows. Let $\eta$ be the same as before, $|\bar x|=1/2$ and define 
\begin{align*}
u_{\va} (x)=\eta\left(\frac {x-\bar x} {\va}\right).
\end{align*}
Hence,
\begin{align*}
u_{\va}^{\ast}(x) =\eta\left(\frac x {\va}\right).
\end{align*}
As above, we only need to check the inequality \eqref{eq:aux}. Since $\Omega=B_1$, we have $\Omega^*=\Omega$ and $F=\tilde F$. Thus,  by change of variables and noticing the support of $\eta$, this reduces to checking the inequality
\begin{align}\label{eq:auxbarx}
\int_{B_1} \eta^p(x) F(\bar x+\va x) dx > \int_{B_1} \eta^p(x) F(\va x) dx.
\end{align}
Since
\begin{align*}
F(\bar x)=\int_{\R^n\setminus B_1} \frac 1 { |\bar x-y|^{n+p \sigma}} dy=\int_{\R^n\setminus B_1(\bar x)} \frac 1 { |z|^{n+p \sigma}} dz>\int_{\R^n\setminus B_1} \frac 1 { |z|^{n+p \sigma}} dz=F(0),
\end{align*}
where we used \eqref{eq:lemma2} in the last inequality (noticing $(B_1(\bar x))^*=B_1$), 
the inequality \eqref{eq:auxbarx} holds for all $\va$ sufficiently small by using the Lebesgue dominated convergence theorem. 
\end{proof}

We now give the proof of Theorem \ref{thm:rearrangementestimate}.

\begin{proof}[{Proof of Theorem \ref{thm:rearrangementestimate}}] We only need to consider the case where $\Omega\subset\R^n$ is an open set that satisfies $|\R^n\setminus\Omega|>0$. Let  $u\in C^\infty_c(\Omega)$ be a nonnegative function. 

Then
\begin{align}
&\iint_{\R^n\times \R^n} \frac{|u^*(x)-u^*(y)|^p}{|x-y|^{n+\sigma p}}\,\ud x\ud y \nonumber\\
&\le \iint_{\R^n\times \R^n} \frac{|u(x)-u(y)|^p}{|x-y|^{n+\sigma p}}\,\ud x\ud y \nonumber\\
&=\iint_{\Omega\times \Omega} \frac{|u(x)-u(y)|^p}{|x-y|^{n+ \sigma p}}\,\ud x\ud y+2\int_{\Omega}u^p(x)\left(\int_{\R^n\setminus \Omega}\frac{1}{|x-y|^{n+\sigma p}}\,\ud y\right)\ud x. \label{eq:aux2}
\end{align}
As in Loss-Sloane \cite{LS} and Dyda-Frank \cite{DF}, we denote
\[
d_\omega(x)=\inf\{|t|: x+t\omega\not\in\Omega\},\quad x\in\R^n,\quad \omega\in\mathbb{S}^{n-1},
\]
where $\mathbb{S}^{n-1}$ is the $(n-1)$-dimensional sphere, and 
\[
m_{\alpha}(x)=\left(\frac{2\pi^{\frac{n-1}{2}} \Gamma(\frac{1+\alpha}{2})}{\Gamma(\frac{N+\alpha}{2})}\right)^{\frac{1}{\alpha}} \left(\int_{\mathbb{S}^{n-1}}\frac{1}{d_{\omega}(x)^\alpha}\,\ud \omega\right)^{-\frac{1}{\alpha}}.
\]
Then we have
\begin{align*}
\int_{\R^n\setminus \Omega}\frac{1}{|x-y|^{n+\sigma p}}\,\ud y\le \int_{\mathbb{S}^{n-1}}\ud \omega \int_{d_{\omega}(x)}^{\infty}\frac{1}{r^{n+\sigma p}}\,\ud r&=(n+\sigma p-1)\int_{\mathbb{S}^{n-1}} \frac{1}{d_\omega(x)^{\sigma p}}\ud \omega\\
&=\frac{C(n,\sigma, p)}{(m_{\sigma p}(x))^{\sigma p}}
\end{align*}
for some constant $C(n,\sigma, p)$ depending only on $n,\sigma$ and $p$, but \emph{not} on $\Omega$. Thus, we have
\begin{align}
\int_{\Omega}u^p(x)\left(\int_{\R^n\setminus \Omega}\frac{1}{|x-y|^{n+\sigma p}}\,\ud y\right)\ud x
&\le C(n,\sigma, p) \int_{\Omega}\frac{u^p(x)}{(m_{\sigma p}(x))^{\sigma p}} \,\ud x\nonumber\\
&\le C(n,\sigma, p) \iint_{\Omega\times \Omega} \frac{|u(x)-u(y)|^p}{|x-y|^{n+ \sigma p}}\,\ud x\ud y,\label{eq:auxsobolev}
\end{align}
where we use Theorem 1.2 (fractional Hardy inequality) of Loss-Sloane \cite{LS} in the last inequality. Therefore, combining \eqref{eq:aux2} and \eqref{eq:auxsobolev},  we have
\[
\iint_{\R^n\times \R^n} \frac{|u^*(x)-u^*(y)|^p}{|x-y|^{n+\sigma p}}\,\ud x\ud y\le C(n,\sigma, p) \iint_{\Omega\times \Omega} \frac{|u(x)-u(y)|^p}{|x-y|^{n+ \sigma p}}\,\ud x\ud y.
\]
Theorem \ref{thm:rearrangementestimate} is proved.
\end{proof}

\begin{rem*}
The above proof of Theorem \ref{thm:rearrangementestimate} can be used to prove \eqref{eq:FSI}. Indeed, if $n\ge 2$, $\sigma\in (0,1)$ and $1<\sigma p<n$, then for every open set $\Omega\neq\R^n$ and all $u\in C^\infty_c(\Omega)$, we have
\begin{align*}
\left(\int_{\Omega}|u(x)|^{\frac{np}{n-\sigma p}}\,\ud x\right)^{\frac{n-\sigma p}{n}}&=\left(\int_{\R^n}|u(x)|^{\frac{np}{n-\sigma p}}\,\ud x\right)^{\frac{n-\sigma p}{n}}\\
&\le C(n,\sigma,p) \iint_{\R^n\times \R^n} \frac{|u(x)-u(y)|^p}{|x-y|^{n+\sigma p}}\,\ud x\ud y\\
&\le C(n,\sigma,p) \iint_{\Omega\times \Omega} \frac{|u(x)-u(y)|^p}{|x-y|^{n+\sigma p}}\,\ud x\ud y, 
\end{align*}
where in the first inequality we used the classical fractional Sobolev inequality in $\R^n$, and in the second inequality we used \eqref{eq:aux2} and \eqref{eq:auxsobolev}.
\end{rem*}

\bibliographystyle{abbrv}

\bigskip

\noindent D. Li

\noindent Department of Mathematics, The Hong Kong University of Science and Technology\\
Clear Water Bay, Kowloon, Hong Kong\\[1mm]
Email: \textsf{madli@ust.hk}

\bigskip

\noindent K. Wang

\noindent Department of Mathematics, The Hong Kong University of Science and Technology\\
Clear Water Bay, Kowloon, Hong Kong\\[1mm]
Email: \textsf{kewang@ust.hk}

\end{document}